\documentclass[11pt,reqno,a4paper]{amsart}

\usepackage[utf8]{inputenc}

\usepackage{amsthm,amsmath,amsfonts,amssymb}
\numberwithin{equation}{section}  
\usepackage{mathrsfs}
\usepackage{todonotes}

\usepackage{tikz-cd}

\usepackage{mathtools}	 
\usepackage{bm}	  
\usepackage{bbm}	 

\usepackage{tensor}	 
\usepackage[colorlinks=true]{hyperref} 
\usepackage[all]{hypcap}       
 
\newcommand{\C}{\mathbb{C}}

\newcommand{\R}{\mathbb{R}}

\newcommand{\T}{\mathbb{T}}
\newcommand{\bB}{\mathbf{B}}
\newcommand{\bL}{\mathbf{L}}
\newcommand{\bT}{\mathbf{T}}

\newcommand{\Z}{\mathbb{Z}}
\newcommand{\ii}{\operatorname{i}}
\newcommand{\eps}{\varepsilon}

\newcommand{\opZ}{\operatorname{Z}}
\newcommand{\ch}{\operatorname{ch}}

\newcommand{\vol}{\operatorname{vol}}

\newcommand{\Rea}{\operatorname{Re}}
\newcommand{\Imm}{\operatorname{Im}}

\newcommand{\del}{\partial}
\newcommand{\delbar}{\bar{\partial}}
\newcommand{\PP}{\mathbb{P}}
\newcommand{\LL}{\mathbb{L}}

\newcommand{\cA}{\mathcal{A}}

\newcommand{\cX}{\mathcal{X}}
\newcommand{\chX}{\check{X}}
\newcommand{\chZ}{\check{Z}}

\newcommand{\chcalZ}{\check{\mathcal{Z}}}
\newcommand{\chC}{\check{C}}
\newcommand{\cY}{\mathcal{Y}}
\newcommand{\cL}{\mathcal{L}}
\newcommand{\chf}{\check{f}}
\newcommand{\chs}{\check{\sigma}}

\newcommand{\chOm}{\check{\Omega}}
\newcommand{\chom}{\check{\omega}}
\newcommand{\chB}{\check{\mathbf{B}}}

\newcommand{\cE}{\mathcal{E}}
\newcommand{\cF}{\mathcal{F}}

\newcommand{\cP}{\mathcal{P}}

\newcommand{\cZ}{\mathcal{Z}}

\newcommand{\olo}{\mathcal{O}}
\newcommand{\cI}{\mathcal{I}}

\newcommand{\PD}{\operatorname{PD}}

\newcommand{\Spec}{\operatorname{Spec}}

\newcommand{\pd}{\operatorname{PD}}

\newcommand{\DFuk}{D\!\operatorname{Fuk}}
\newcommand{\DCoh}{D\!\operatorname{Coh}}

\newcommand{\Coh}{\operatorname{Coh}}

\newcommand{\Pic}{\operatorname{Pic}}
\newcommand{\NS}{\operatorname{NS}}
\newcommand{\Amp}{\operatorname{Amp}}

\newcommand{\pow}[1]{[\![ {#1} ]\!]}

\newtheorem{thm}{Theorem}[section]
 
\newtheorem{prop}[thm]{Proposition}
\newtheorem{lemma}[thm]{Lemma}
\newtheorem{cor}[thm]{Corollary}

\theoremstyle{definition}
\newtheorem{definition}[thm]{Definition}

\theoremstyle{remark}

\newtheorem{rmk}[thm]{Remark}

\title[A Thomas-Yau-Joyce result for Lagrangian spheres]{A Thomas-Yau-Joyce result for Lagrangian spheres in K3 surfaces}
\author{Jacopo Stoppa}
 
\date{September 30, 2026}

\begin{document}

\maketitle

\begin{abstract} We consider objects in the Fukaya category of a class of K3 surfaces defined by certain Lagrangian spheres. Assuming that homological mirror symmetry for K3 surfaces holds  with sufficiently strong (expected) properties, we prove that, in this case, stability with respect to a suitable Bridgeland stability condition implies the existence of an isomorphic \emph{special} Lagrangian sphere, as predicted by the general Thomas-Yau-Joyce conjectures. In particular this holds unconditionally for suitable quartic surfaces in $\PP^3$ or sextics in $\PP(3,1,1,1)$ and their mirrors. A variant holds on the Calabi-Yau threefolds obtained by taking the product of our K3 surfaces with an elliptic curve. These seem to be the first results of this type on compact manifolds.
\end{abstract}

\section{Introduction}\label{Intro}

Let $(\chX, \chom, \chOm)$ denote a K3 surface endowed with a Ricci flat K\"ahler form and a holomorphic volume form (the reason for our notation will become clear in a moment). A (real) $\chom$-Lagrangian submanifold $L \subset \chX$ is called \emph{special Lagrangian (sLag)} if the restriction $\chOm|_{L}$ has constant phase; equivalently, $L$ is a Harvey-Lawson calibrated submanifold with respect to (some rotation of) $\chOm$ (see \cite{HarveyLawsonCalibrated}). 

Suppose $L \subset \chX$ is a fixed embedded Lagrangian sphere. In the present work we study the problem of whether there exists a \emph{special Lagrangian (sLag) sphere} $\tilde{L} \subset \chX$ such that \emph{$L$, $\tilde{L}$ define the same object in the Fukaya category} $\DFuk(\chX, \chom)$. This is a special case of the general Thomas-Yau-Joyce conjectures \cite{Joyce_ThomasYau, Thomas_MomentMirror, ThomasYau} for Calabi-Yau manifolds in all dimensions. In turn it generalises the problem of constructing such special Lagrangians within a fixed Hamiltonian isotopy class on K3 surfaces (studied e.g. by Schoen and Wolfson \cite{SchoenWolfson}). It is expected that there should exist a stability condition $\chs$ on $\DFuk(\chX, \chom)$ in the sense of Bridgeland such that the (unique) sLag $\tilde{L}$ exists if and only if the object in $\DFuk(\chX, \chom)$ defined by $L$ is $\chs$-stable. The case of Lagrangian spheres in K3 surfaces is discussed in particular in \cite{ThomasYau}, Section 5.3. 

Here we focus on showing the implication \emph{Bridgeland stability $\Rightarrow$ sLag} for certain Lagrangian spheres on suitable K3 surfaces, as well as products of these K3s with an elliptic curve. Briefly, in our examples, we show that one can avoid more standard but difficult analytic methods (such as the Lagrangian mean curvature flow or glueing) by assuming instead that homological mirror symmetry for K3 surfaces holds with sufficiently strong (expected) properties; see Remark \ref{ProofRmk} for some more details on the method of proof. In particular our result holds unconditionally for suitable quartic surfaces in $\PP^3$, sextics in $\PP(3,1,1,1)$ and their mirrors, and seems to give the first example of the principle \emph{Bridgeland stability $\Rightarrow$ sLag} holding on \emph{compact} manifolds. 

\begin{rmk} It is shown in \cite{ThomasYau}, Section 5.3 that there exists a unique Fukaya equivalence class within a fixed homology class of Lagrangian spheres on a K3 surface admitting a sLag representative: namely, the class corresponding to the unique holomorphic representative in a hyperk\"ahler rotation. So the problem we study is precisely that of \emph{characterising this sLag class in terms of Bridgeland stability}.
\end{rmk}
Let us describe our main result informally (all the details are provided in Sections \ref{GrossWilsonSec} and \ref{HMSSec}).  We consider the case when 
\begin{equation*}
(\chX_s, \chB_s + \ii\chom_s, \chOm_s) 
\end{equation*}
is the classical (Hodge-Theoretic) mirror of a K3 surface 
\begin{equation*}
(X, \bB_s + \ii s \omega, \Omega),\,s\gg 0
\end{equation*}
in the sense discussed by Gross-Wilson in \cite{GrossWilson}, Section 1 (based on Dolgachev's approach \cite{DolgaPolarMirror}). Here $X$ is endowed with a \emph{sufficiently general K\"ahler class} $\omega$, $s \gg 0$ is the \emph{large volume/complex structure limit}, and the construction depends on a B-field class $\bB_s$, as well as a suitable choice of a hyperbolic sublattice $H = \langle \sigma_0, E \rangle \subset H^2(X, \Z)$.

Note however that our choices of $H$ and of the B-field will be very special, namely, we will assume that $\sigma_0$ is of type $(1,1)$ and that there is a lift $\hat{\bB}_s = s \hat{\bB} \in H^{1, 1}(X, \R)$ which is \emph{linear} in the large volume/complex structure parameter $s$ (we will actually make a unique distinguished choice of such a lift, following Gross-Wilson \cite{GrossWilson}, Section 1; see Definition \ref{AdmissDef}). This linear regime is the ``large scaling limit" studied in \cite{J_toricThomasYau}, Section 6 and more generally in \cite{YWFan_stability}. The complex structure of $X$ is also chosen generally with respect to these fixed properties.   

\begin{thm}[Theorem \ref{MainThm}] Choose our pair $X, \chX_s$ as above. Suppose that there is a fixed homological mirror symmetry equivalence 
\begin{equation*}
\psi^*\DCoh(\cX_{\xi}) \cong \DFuk(\chX_s, \chB_s + \ii \chom_s)
\end{equation*}
satisfying sufficiently strong, expected properties (here $\cX_{\xi}$ denotes the generic fibre of the mirror family for $(\chX_s, \chB_s + \ii \chom_s)$ and $\psi$ is the mirror map evaluated at the K\"ahler parameters, see Section \ref{HMSSec}). Let us denote by 
\begin{equation*}
\cL_s \in \DFuk(\chX_s, \chom)        
\end{equation*}
the object in the Fukaya category mirror to the structure sheaf $\olo_{\cX_{\xi}}$. Then $\cL_s$ can be represented by a Lagrangian sphere, and there is a sequence of Bridgeland stability conditions $\chs_s$ on $\DFuk(\chX_s, \chom_s)$ such that, if $\cL_s$ is $\chs_s$-stable for $s \gg 0$, $\cL_s$ can be represented by a (unique, smooth) \emph{special} Lagrangian sphere. This holds unconditionally for suitable quartic surfaces in $\PP^3$, sextics in $\PP(3,1,1,1)$ and their mirrors.
\end{thm}
By the arguments of \cite{ThomasYau}, Section 6.3, we expect that the object $\cL_s \in \DFuk(\chX_s, \chom)$ is \emph{not} always isomorphic to a \emph{special} Lagrangian sphere, as the parameter $\hat{\bB}$ in our construction varies (although we do not have a concrete example at the moment). This expectation is also compatible with the more analytic viewpoint sketched in Section \ref{dHYMSec} (see Remark \ref{ProofRmk} below).

We cannot prove the converse implication \emph{sLag $\Rightarrow$ stability} at present. Under certain assumptions, Li \cite{YangLi_ThomasYau} shows that compact sLags in any Stein Calabi-Yau manifold must satisfy a condition reminiscent of Bridgeland stability. Lotay and Oliveira \cite{LotayOliv_MCF, Lotay_GibbonsHawkingSLag}, building on the work of Lotay, Schulze and Sz\'ekelydidi \cite{LotaySchulzeSzeke}, show that a form of the Thomas-Yau-Joyce conjectures (not directly involving Bridgeland stability) holds on the non-compact, complete Calabi-Yau surfaces with circle symmetry obtained by the Gibbons-Hawking ansatz. Other recent results and proposals concerning the Thomas-Yau-Joyce conjectures include \cite{YLShen_K3FibredSLag}, Section 6.3 and \cite{HaidenKatzKontPand}, Section 3.5. The works of the author \cite{J_toricThomasYau, J_sLagStability, StoppaSmoothings} and Fan \cite{FanSLags} show the implication \emph{Bridgeland stability $\Rightarrow$ sLag} for suitable non-compact Lagrangians in examples of non-compact almost Calabi-Yau manifolds (i.e. some Landau-Ginzburg models). 

\begin{rmk} We emphasise that $\chX_s$ is \emph{never} obtained as a hyperk\"ahler rotation of $X$ in our construction (it becomes asymptotically close to a hyperk\"ahler rotation of $X$ as $s \to +\infty $ only in the very special case when $\bB = 0$). The fact that the mirror of a K3 surface is not given by a hyperk\"ahler rotation in general follows from the discussion in \cite{GrossWilson}, Section 1 (recalled in our Section \ref{GrossWilsonSec}), and Gross discusses the point explicitly in \cite{GrossMO}. This is of crucial importance for our construction. 
\end{rmk}
Our argument provides a reduction from a \emph{differential-geometric} problem to \emph{symplectic/algebro-geometric} properties of mirror symmetry on K3 surfaces, which are known at least in some cases (i.e. at least for suitable quartic surfaces in $\PP^3$, sextics in $\PP(3,1,1,1)$ and their Batyrev mirrors, see Proposition \ref{MirrorProp}). Note that the Thomas-Yau-Joyce conjecture predicts that a reduction from differential-geometric objects (sLags) to symplectic/algebro-geometric (Bridgeland stable) objects is \emph{always} possible; here we show how it can be achieved at least for some very special objects on very special compact Calabi-Yau manifolds.

\begin{rmk}[Method of proof]\label{ProofRmk} The main virtue of our argument is that it shows clearly the role of Bridgeland stability on the Fukaya category in this case. As we will see (Corollary \ref{StabToFamilyCor}), the key point is that Bridgeland stability guarantees the positivity of a suitable twist of $\omega$ by the section $\sigma_0$ and the B-field lift $\hat{\bB}$ (this is known as \emph{twisted ampleness}). Through the Torelli theorem, this twisted ampleness yields the existence of a suitable \emph{family} $\cZ$ of K3 surfaces, with a special fibre $\cZ_1$ isomorphic to $(\chX_s, \chB_s + \ii\chom_s, \chOm_s)$, and a further special fibre $\cZ_0$, on which a sLag with required properties exists (this is summarised in the diagrams \eqref{Diagram1}, \eqref{Diagram2}). A sufficiently strong form of homological mirror symmetry then allows to deform this solution from $\cZ_0$ to $\cZ_1 = \chX$.

The notion of twisted ampleness appears in the theory of deformed Hermitian Yang-Mills (dHYM) connections, going back to the work of Jacob and Yau \cite{JacobYau_special_Lag} (see e.g. \cite{CollinsLoYau_K3} for a study in the case of K3 surfaces; from our current viewpoint this should be read in conjunction with the results of \cite{YWFan_stability} and \cite{J_toricThomasYau}, Section 6). Although our proof of Theorem \ref{MainThm} does \emph{not} use dHYM connections, it seems that it would have been hard to arrive at the construction of the family above without relying on the motivation they provide. We sketch the dHYM approach (or rather the analytic difficulties it entails) in Section \ref{dHYMSec}. 
\end{rmk}

\begin{rmk} Although the general Thomas-Yau-Joyce conjectures are formulated independently of mirror symmetry, the latter provided essential motivation for Thomas' original proposal \cite{Thomas_MomentMirror}. In \cite{Thomas_MomentMirror}, Section 3, Thomas discusses the case of K3 surfaces, and makes the remarkable observation that a hyperk\"ahler rotation of Donaldson's moment map picture for the $J$-equation (see \cite{Donaldson_momentmaps_diffeomorphisms}) leads to a genuine (infinite dimensional) moment map picture for special Lagrangian embeddings. We believe that this viewpoint is closely related to the approach presented here, especially since the $J$-equation on surfaces is essentially equivalent to the dHYM equation (more precisely, the $J$-equation and the dHYM equation are both equivalent to complex Monge-Amp\`ere equations).
\end{rmk}
Our main result Theorem \ref{MainThm} can be used to construct examples of objects in the Fukaya category of some compact Calabi-Yau threefolds satisfying the same Thomas-Yau-Joyce principle: the construction depends on parameters, such that the objects are (expectedly) not always isomorphic to sLags, and a sufficient condition for the existence of such sLags is given by stability with respect to a suitable (weak) Bridgeland stability condition.

We consider the case when $\chZ_s$ is the product Calabi-Yau threefold, \emph{depending on $s >0$}, given by
\begin{align*}
(\chZ_s, \omega_{\chZ_s}, \Omega_{\chZ_s}) &:= (\chX_s \times \chC, \chB_s + \ii(p^* \chom_s + q^*\omega_{\chC}), p^* \chOm_s \wedge q^*\Omega_{\chC}),
\end{align*}
where $(\chC, \omega_{\chC}, \Omega_{\chC})$ denotes an elliptic curve and $p$, $q$ are the projections. For any choice of line bundle $\cE_C$ on the mirror $C$ of $\chC$ and for $s \geq 1$, we denote by  
\begin{equation*}
\cL_{\chZ_s} \in \DFuk(\chZ_s, \omega_{\chZ_s})
\end{equation*} 
the object mirror to 
\begin{equation*}
p^*(\olo_{\cX_{\xi}}) \otimes q^*(\cE_C) \in \Pic(Z). 
\end{equation*}
\begin{thm}[Theorem \ref{ProductThm}] Suppose that homological mirror symmetry for the product $\chZ_s = \chX_s \times \chC$ holds with sufficiently strong (expected) properties. There exists a sequence of weak Bridgeland pre-stability conditions $\chs_{C, s}$ on $\DFuk(\chZ_s, \omega_{\chZ_s})$ such that if $\cL_{\chZ_s}$ is $\chs_{C, s}$-stable for $s \gg 0$, then $\cL_{\chZ_s}$ is isomorphic to a smooth special Lagrangian, homeomorphic to $S^2 \times S^1$. 

This holds unconditionally for $X, \chX_s$ given by suitable quartic surfaces in $\PP^3$, sextics in $\PP(3,1,1,1)$ and their mirrors.
\end{thm}
Theorem \ref{ProductThm} follows from Theorem \ref{MainThm} by applying the construction of stability conditions on products with curves due to Liu \cite{Liu_products}. Note that Liu constructs genuine stability conditions, not just weak ones. It would be interesting to see if Theorem \ref{ProductThm} also holds for these genuine stability conditions.

Section \ref{GrossWilsonSec} contains some background on mirror symmetry for K3 surfaces following Gross-Wilson \cite{GrossWilson}. Section \ref{HMSSec} explains the required properties from homological mirror symmetry and provides a precise statement of our main result, Theorem \ref{MainThm}. The latter is proved in Section \ref{MainSec}. Our examples on products are discussed in Section \ref{ProductSec}, where Theorem \ref{ProductThm} is proved. Section \ref{dHYMSec} contains a brief and conjectural discussion of the alternative analytic approach using dHYM connections in conjunction with the results of Gross-Wilson \cite{GrossWilson} on the SYZ conjecture for K3 surfaces.\\

\noindent{\textbf{Acknowledgements.}} I am grateful to Mahmoud Elimam, Yu-Wei Fan, Rah\'ul Gonz\'alez Molina, Dominic Joyce, Ailsa Keating, Yu-Shen Lin, Jason Lotay, Emanuele Macr\`i, Mirko Mauri, Helge Ruddat, Paolo Stellari and  Richard Thomas for helpful conversations.\\

\noindent{\textbf{Note on LLMs.}} Large language models were not used for the present work.
  
\section{Gross-Wilson setup}\label{GrossWilsonSec}  Our starting point is the description of a certain relation between Hodge-theoretic mirror symmetry and hyperk\"ahler rotation on K3 surfaces given by Gross and Wilson in \cite{GrossWilson}, Section 1, building on the results of Dolgachev \cite{DolgaPolarMirror}. Let us recall their construction (see also \cite{GrossFibrations2}, Section 7), following their notation closely. 

We fix a K3 surface $X$ with K3 lattice $\bL := H^2(X, \Z)$ and a choice of elements $E,\,\sigma_0 \in \bL$ with 
\begin{equation*}
E^2 = 0,\,\sigma^2_0 = -2,\,E.\sigma_0 = 1,
\end{equation*}
thus generating a copy of the hyperbolic lattice $H$ contained in $\bL$. 

Classical \emph{Hodge-theoretic} mirror symmetry for K3 surfaces is an involution on the set of triples 
\begin{equation*}
(X, \bB + \ii \omega, \Omega)
\end{equation*}
where $X$ is a K3 surface, $\omega$ is a K\"ahler class, $\bB$ is a B-field, and $\Omega$ is the class of a holomorphic $2$-form, satisfying
\begin{equation*}
\omega \in E^{\perp}\otimes \R,\,\bB \in E^{\perp}/E \otimes \R
\end{equation*}
($(-)^{\perp}$ is taken with respect to the intersection pairing). Morally, $E$ represents the class of a fibre of a special Lagrangian fibration (with respect to which we compute mirrors of $X$), with a \emph{smooth} (not necessarily Lagrangian) section $\sigma_0$.

One further assumes the normalisation conditions (on cohomology classes)
\begin{equation*}
\Imm \Omega \in E^{\perp} \otimes \R,\,\omega^2 = (\Rea \Omega)^2 = (\Imm \Omega)^2.
\end{equation*}
Then the involution can be described explicitly by the relations
\begin{align}\label{ClassicalMirrorMap}
\nonumber\chOm &= (E. \Rea \Omega)^{-1}(\sigma_0 + \bB + \ii \omega)\mod E\\
\nonumber\chB &= (E. \Rea \Omega)^{-1} \Rea \Omega - \sigma_0\mod E\\
\chom &= (E. \Rea \Omega)^{-1} \Imm \Omega \mod E,
\end{align}
together with
\begin{align*}
& \chom^2 = (\Rea \chOm)^2 = (\Imm \chOm)^2,\\
& \chom. \Rea \chOm = \chom. \Imm \chOm = (\Rea \chOm).(\Imm \chOm) = 0. 
\end{align*} 

These relations fix the actual classes $\chom$, $\chOm$, not just their mod $E$ reductions. Gross-Wilson observed that these have a nice expression in terms of a certain preferred lift of the B-field class (from $E^{\perp}/E \otimes \R$ to $E^{\perp} \otimes \R$).
\begin{lemma}[\cite{GrossWilson}, Section 1]\label{GrossWilsonLem} Let $\hat{\bB}$ denote the unique lift of $\bB$ to $E^{\perp} \otimes \R$ such that $\hat{\bB} . \sigma_0 = 0$. Then we have 
\begin{align*}
&\chOm = (E.\Rea\Omega)^{-1}\left(\sigma_0 + \hat{\bB} + \ii \omega + \left(\frac{ \omega^2 - \hat{\bB}^2}{2} + 1 - \ii \omega.(\sigma_0 + \hat{\bB})\right)E\right),\\ 
&\chom = (E.\Rea\Omega)^{-1}\left(\Imm \Omega - (\Imm \Omega . (\sigma_0 + \hat{\bB})) E \right).
\end{align*}
\begin{proof} The claim follows straightforwardly from the identities
\begin{equation*}
\chOm^2 = \chom. \chOm = 0,
\end{equation*}
which in turn are implied by the normalisation conditions
\begin{align*}
& \chom^2 = (\Rea \chOm)^2 = (\Imm \chOm)^2,\\
& \chom. \Rea \chOm = \chom . \Imm \chOm = (\Rea \chOm).(\Imm \chOm) = 0. 
\end{align*} 
\end{proof}
\end{lemma}  
\begin{rmk}\label{NormalisationRmk} By the normalisation condition $\omega^2 = (\Rea \Omega)^2 = (\Imm \Omega)^2$, scaling $\omega$ by $\mu > 0$ induces the same scaling $\Rea \Omega \mapsto \mu \Rea \Omega$, $\Imm \Omega \mapsto \mu\Imm \Omega$. Thus, up to replacing $\omega$ by $\mu \omega$ for appropriate $\mu$ (which is irrelevant for our purposes), we can assume the further normalisation condition 
\begin{equation*}
E.\Rea\Omega = 1.
\end{equation*}
\end{rmk}
\section{Mirror symmetry and statement of the main result}\label{HMSSec}
Fix $(X, s(\bB + \ii \omega), \Omega)$, with sublattice $H = \langle \sigma_0, E \rangle \subset \bL$, depending on a scaling parameter $s > 0$. Let $(\chX_s, \chB_s + \ii \chom_s, \chOm_s)$ denote its Hodge-theoretic mirror. 
\begin{definition}\label{AdmissDef} We say that a pair $X, \chX_s$ with sublattice $H$ as above is \emph{admissible} if $\omega$, $\sigma_0$ as well as the normalized lift $\hat{\bB}$ of $\bB$ in the sense of Lemma \ref{GrossWilsonLem} belong to $\NS(X)_{\R} = \Pic(X)\otimes \R$. (Note that this only depends on properties of $(X, \bB + \ii \omega, \Omega)$ and the sublattice $H$, although we think of it as a property of the pair $X, \chX_s$).
\end{definition}
\begin{definition} A \emph{homological mirror} for $(\chX_s, \chB_s + \ii \chom_s, \chOm_s)$ is a formal family of K3 surfaces $\cX \to \Spec \C\pow{q}$, such that there is an equivalence
\begin{equation*}
\psi^*\DCoh(\cX_{\xi}) \cong \DFuk(\chX_s, \chB_s + \ii \chom_s)
\end{equation*}
where $\psi$ denotes a suitable automorphism of $\C\pow{q}$ (the mirror map, evaluated at the K\"ahler parameters $\chB_s + \ii \chom_s$). 
\end{definition} 
Our main argument is conditional on a sufficiently strong version of homological mirror symmetry. Let us spell out the properties we shall need. 
\begin{definition}[SYZ-type fibrations]\label{SYZtypeDef} Suppose $(Z, B_Z + \ii\omega_Z, \Omega_Z)$ is a K3 surface and $\cY$ is a K3 surface over $\C\pow{q}$, satisfying the mirror equivalence
\begin{equation*}
\DFuk(Z, B_Z + \ii\omega_Z) \cong \psi^*\DCoh(\cY_{\xi}).
\end{equation*}
A \emph{homological SYZ fibration} is a Lagrangian fibration $\pi\!: Z \to S^2$ inducing a bijection
\begin{align*}
\Pic(\cY_{\xi}) &\xleftrightarrow{1:1} \{\textrm{Lagrangian sections of } \pi\}/\sim,\\
\cE &\mapsto \cL(\cE), 
\end{align*} 
compatible with the fixed mirror equivalence, where the relation $\sim$ denotes fibre-preserving Hamiltonian isotopy. In particular, there is a well defined class $[\cL(\cE)] \in H_2(Z, \Z)$. 
\begin{rmk}
This notion is modelled on the results of Hacking-Keating \cite{HackKeatingK3} (see in particular \cite{HackKeatingK3}, Theorem 4.2 and Proposition 4.23).
\end{rmk}
\end{definition}
\begin{definition}[Hodge/HMS compatibility]\label{SYZtypeProperties} Suppose we are given an admissible Hodge-theoretic mirror pair in the sense of Definition \ref{AdmissDef},  
\begin{equation*}
(X, s(\bB + \ii \omega), \Omega),\,(\chX_s, \chB_s + \ii \chom_s, \chOm_s),
\end{equation*}
with sublattice $H = \langle \sigma_0, E \rangle \subset \bL$, and a mirror equivalence  
\begin{equation*}
\psi^*\DCoh(\cX_{\xi}) \cong \DFuk(\chX_s, \chB_s + \ii \chom_s), 
\end{equation*}
which admits a corresponding homological SYZ fibration $\pi\!: \chX_s \to S^2$ in the sense of Definition \ref{SYZtypeDef}. This induces an inclusion 
\begin{equation*}
\iota\!:\Pic(\cX_{\xi}) \hookrightarrow H_2(\chX_s, \Z) = \bL, 
\end{equation*}
the latter using our fixed identification. We say that \emph{homological mirror symmetry for $\chX_s$ holds compatibly with Hodge-theoretic mirror symmetry} if 
\begin{enumerate}
\item[$(i)$] there is an isomorphism $i\!:\Pic(X) \cong \Pic(\cX_{\xi})$, preserving the ample cones, such that $\iota\circ i$ is the identity under our fixed identification $H_2(\chX, \Z) = \bL$;
\item[$(ii)$] the fibre class of the homological SYZ fibration $\pi\!: \chX_s \to S^2$ is given by $E$;
\item[$(iii)$] the \emph{Gamma property} holds: denoting by $[\cL(\olo_{\cX_{\xi}})] \in H_2(\chX_s, \Z)$ the class represented by the equivalence class of Lagrangian spheres $\cL(\olo_{\cX_{\xi}})$, we have 
\begin{equation*}
\int_X e^{ u \omega - \ii u\hat{\bB}}  = \int_{[\cL(\olo_{\cX_{\xi}})]} \chOm_s + O(s^{-1})   
\end{equation*} 
for $s\gg 0$, where $u$, $s$ are related by a base-change $s = u^{r}$ for some $r > 0$.
\end{enumerate}
\end{definition}
\begin{definition}\label{StructureObject} Suppose that $(X, s(\bB + \ii \omega), \Omega)$ has sublattice $H = \langle \sigma_0, E \rangle \subset \bL$ and that homological mirror symmetry holds compatibly with Hodge-theoretic mirror symmetry in the sense of Definition \ref{SYZtypeProperties} $(i)$, $(ii)$. Then we denote by  
\begin{equation*}
\cL_s \in \DFuk(\chX_s, \chB_s+\ii\chom_s)
\end{equation*} 
the object mirror to the structure sheaf  $\olo_{\cX_{\xi}} \in \Pic(\cX_{\xi})$.
\end{definition}

\begin{thm}\label{MainThm} Fix an \emph{admissible} Hodge-theoretic mirror pair $X, \chX_s$ with sublattice $H$ (see Definition \ref{AdmissDef}), with general K\"ahler class $\omega$ and holomorphic volume form $\Omega$ (i.e. lying away from a countable union of analytic subvarieties). Suppose that homological mirror symmetry for $\chX_s$ holds compatibly with Hodge-theoretic mirror symmetry in the sense of Definition \ref{SYZtypeProperties}. Then, for all $s \gg 0$,
\begin{enumerate}
\item[$(i)$] the object $\cL_s \in \DFuk(\chX_s, \chB_s+\ii\chom_s)$ mirror to $\olo_{\cX_{\xi}}$ is represented by a smooth Lagrangian sphere;
\item[$(ii)$] there exist stability conditions $\chs_s$ on $\DFuk(\chX_s, \chB_s+\ii\chom_s)$ such that, if the object $\cL_s$ is $\chs_s$-stable for $s \gg 0$, $\cL_s$ can be represented by a (unique, smooth) \emph{special} Lagrangian sphere.
\end{enumerate}
By known mirror symmetry results (summarised in Proposition \ref{MirrorProp}), the required properties hold in particular when $X, \chX_s$ are given by quartic surfaces in $\PP^3$ or sextics in $\PP(3,1,1,1)$ and their Batyrev mirrors, provided the following additional conditions hold:
\begin{enumerate}
\item[$(a)$] $\hat{\bB}$ is induced by a class in $\NS(\widehat{\PP})_{\R}$, where $\widehat{\PP}$ denotes $\PP^3$ or a crepant resolution of $\PP(3,1,1,1)$;
\item[$(b)$] the Gross-Wilson normalised holomorphic volume is such that $\Rea \Omega$ is integral and $\chom_s$ is also the pullback of a real Neron-Severi class from the ambient.
\end{enumerate}
\end{thm}

Note that, according to the discussion in \cite{ThomasYau}, Section 6.3, there might not be a special Lagrangian representative of the class $\cL_s \in \DFuk(\chX_s, \chom_s)$. From our current viewpoint, although there exists a (generically smooth) unique special Lagrangian sphere in the homology class $[\cL_s]$, it might not be equivalent in $\DFuk(\chX_s, \chom_s)$ to a section of the homological SYZ fibration (i.e. morally ``not Hamiltonian isotopic to such a section"). So our main result shows that the ``large scaling"  Bridgeland stability of the Fukaya class represented by a Lagrangian sphere in $\cL_s$ is a sufficient condition for the existence of a \emph{special} Lagrangian sphere in the same class.

Let us summarise part of what is currently known concerning the properties of K3 mirror symmetry spelled out in Definition \ref{SYZtypeDef}, thanks the deep results of Sheridan-Smith \cite{SheridanSmithGP} and Hacking-Keating \cite{HackKeatingK3} building on the foundational work of Seidel \cite{SeidelQuartic}, Ganatra-Pardon-Shende \cite{GaPaSh_Liouville}, Lekili-Ueda \cite{LekiliUeda} and several other authors. 

\begin{prop}\label{MirrorProp} The following general properties of K3 mirror symmetry are already known or conjectured to hold.
\begin{enumerate}
\item[$(i)$] Suppose $(Z, \omega_Z, \Omega_Z)$ is a projective polarised K3 surface (i.e. $\omega_Z$ is integral). Then there exist a projective K3 surface $\cY$ over $\C\pow{q}$, given by a formal smoothing of a type III degeneration, with generic fibre $\cY_{\xi}$, satisfying 
\begin{equation*}
\DFuk(Z, \omega_Z) \cong \psi^*\DCoh(\cY_{\xi}),
\end{equation*} 
and $Z$ admits a homological SYZ fibration $\pi\!: Z \to S^2$ (see Hacking-Keating \cite{HackKeatingK3}, Theorem 1.2, Section 1.4 and Proposition 4.23). This equivalence satisfies $(i)$, $(ii)$ in Definition \ref{SYZtypeProperties} (see \cite{HackKeatingK3}, Section 1.3). 
\item[$(ii)$] In the special case when the pair $(Z, \cY_{\xi})$ is of Greene-Plesser type (i.e. both $Z$ and $\cY_{\xi}$ are resolutions of hypersurfaces in quotients of weighted projective spaces), this equivalence is compatible with the homological mirror symmetry result of Sheridan-Smith \cite{SheridanSmithGP} (see the discussion following the statement of Theorem 1.2 in \cite{HackKeatingK3}), and holds under the more general assumption that $\omega_Z$ is the pullback of a real Neron-Severi class from the ambient.
\item[$(iii)$]  The Gamma property in the sense of Definition \ref{SYZtypeDef} $(ii)$ is conjectured to always hold, and it is known when $(Z, \cY_{\xi})$ is  of Greene-Plesser type and $\hat{\bB}$ is a real Neron-Severi class induces by the ambient (by the results of Sheridan-Smith \cite{SheridanSmithGP}, Section 1.7, combined with those of Abouzaid-Ganatra-Iritani-Sheridan \cite{AGIS}, Theorem C).
\end{enumerate}
Thus, compatibility with Hodge-theoretic mirror symmetry (Definition \ref{SYZtypeProperties}) is conjectured to always hold and known to hold at least when the pair $(Z, \cY_{\xi})$ is of Greene-Plesser type, namely:
\begin{enumerate}
\item a quartic surface in $\PP^3$ and its Batyrev mirror,
\item a sextic in $\PP(3,1,1,1)$ and its Batyrev mirror.
\end{enumerate}  
\end{prop}
\begin{rmk} In work currently in progress (see \cite{KeatingGS25}) Hacking and Keating establish homological mirror symmetry for K3 surfaces in the form of Proposition \ref{MirrorProp} $(i)$ for any K\"ahler form $\omega_Z$ such that the lattice $[\omega_Z]^{\perp} \cap H^2(Z,\Z)$ has a positive sublattice of rank $2$ and contains a copy of the hyperbolic plane. 
\end{rmk}
Let us recall some more details concerning Proposition \ref{MirrorProp}. We fix the data of
\begin{enumerate}
\item[$\bullet$] an algebraic torus $\T_{\C}$, with character and cocharacter lattices $P$ and $P^{\vee}$, of rank $d+1$;
\item[$\bullet$] a reflexive integral polytope $\Delta^{\vee} \subset P^{\vee}_{\R}$ containing the origin;
\item[$\bullet$] a fan $\Sigma$ in $P^{\vee}_{\R}$ giving a star-shaped triangulation of $\Delta^{\vee}$.
\end{enumerate}
Together, these determine a toric variety $\bT_{\Sigma}$ with toric boundary divisor $\del\bT_{\Sigma}$. Similarly, the polar dual $\Delta$ determines a toric variety $\bT^{\vee}_{\Sigma}$.

Fix a function $\theta\!: \Delta^{\vee} \to \R$, denoted by $q \mapsto \theta_q$. Following \cite{AGIS}, Section 1.3, we introduce a polynomial function on  $\T^{\vee}_{\C} \cong (\C^*)^{d+1}$ given by
\begin{equation*}
W_{s, \theta} =  \sum_{q \in \Delta^{\vee}} e^{\ii \theta_q} s^{\alpha(q)} z^q. 
\end{equation*}
The corresponding hypersurface is
\begin{equation*}
F_{s, \theta} := F_{W_s, \theta} := \{W_{s, \theta} = c \} \subset \T^{\vee}_{\C}. 
\end{equation*}

We denote this by $F_s$ when $\theta = 0$. Let $\widehat{\bT}^{\vee}_{\Sigma}$ be a partial crepant resolution of $\bT^{\vee}_{\Sigma}$ which has at worst quotient singularities. Then the hypersurface $F_s \subset \T^{\vee}_{\C}$ compactifies to a quasi-smooth Calabi-Yau hypersurface $\widehat{F}_s \subset \widehat{\bT}^{\vee}_{\Sigma}$.   

We fix the residue volume form
\begin{equation*}
\Omega_{W_s} := \frac{d \log z_0 \wedge \cdots \wedge d\log z_{d}}{d W_s(z)}\big|_{F_{W_s}}. 
\end{equation*}
Let 
\begin{equation*}
C^+_s := F_{s} \cap \T^{\vee}_{\R_+} \subset F_s
\end{equation*}
denote the positive real locus. According to \cite{AGIS}, Section 3.1, $C^+_s$ is homeomorphic to the sphere $S^d$ for sufficiently small $s > 0$.

Fix weights $\nu\!: \Delta^{\vee} \to \Z$, denoted by $q \mapsto \nu_q$. Define a corresponding cycle 
\begin{equation*}
C^{(\nu)}_s \subset F_s
\end{equation*} 
as the parallel transport of $C^+_s \subset F_s$ as we vary $\theta$ continuously from $\theta = 0$ to $\theta = 2\pi \nu$.

Let $H$ denote a quasismooth anticanonical divisor in $\bT^{\vee}_{\Sigma}$. Let 
\begin{equation*}
\widehat{\Gamma}_H := \exp\left(\sum_{k \geq 2}(-1)^k \zeta(k)(k-1)!\ch_k(TH)\right)
\end{equation*}
denote its Gamma class. 
\begin{thm}[\cite{AGIS}, Theorem C]\label{GammaThm} We have
\begin{align*}
\int_{C^{(\nu)}_s \subset F_{s}} \Omega_s = \int_H s^{-\omega} \cdot \widehat{\Gamma}_H\cdot e^{-2\pi \ii \sum_{q \in \del\Delta^{\vee}\cap P^{\vee}}\nu_q D_q} + O(s^{\eps}) 
\end{align*}
as $s\to 0$  for some $\eps > 0$.
\end{thm}
There is a general conjectural argument, provided in \cite{AGIS}, Remark 1.3, which predicts that homological mirror symmetry holds compatibly with Theorem \ref{GammaThm}. This is now fully worked out at least in the special case when \emph{$\bT_{\Sigma}$, $\bT^{\vee}_{\Sigma}$ are quotients of weighted projective spaces}, so the results of Sheridan-Smith on general Greene-Plesser mirrors \cite{SheridanSmithGP} apply at the same time as the results of Hacking-Keating \cite{HackKeatingK3}. 
\begin{thm}[\cite{HackKeatingK3}, \cite{SheridanSmithGP}] Suppose $X := H  \subset \bT_{\Sigma}$ is a K3 surface arising in the construction above, namely, a quartic surface in $\PP^3$ and its Batyrev mirror, as well as a sextic in $\PP(3,1,1,1)$ and its Batyrev mirror. Then, 
\begin{enumerate}
\item[$(i)$] There is a homological mirror symmetry equivalence
\begin{equation*}
\psi^*\DCoh(\cX_{\xi}) \cong \DFuk(\chX_s := \widehat{F}_s),
\end{equation*}
satisfying the properties $(i)$, $(ii)$ in Definition \ref{SYZtypeProperties},
\item[$(ii)$] the corresponding homological SYZ fibration $\pi\!: \chX_s = F_s \to S^2$ is such that the cycle $\cL\left(\sum_{q \in \del\Delta^{\vee}\cap P^{\vee}}\nu_q D_q\right)$ is cohomologous to $C^{(\nu)}_s$. 
\end{enumerate}
\end{thm} 
\section{Proof of the main result}\label{MainSec} 
The Section is devoted to proving Theorem \ref{MainThm}. This requires several steps. For each step we emphasise the role of the general properties of mirror symmetry given in Definition \ref{SYZtypeProperties}, which we always assume to hold throughout. 

We fix throughout a Hodge-theoretic mirror pair
\begin{equation*}
(X, s(\bB + \ii \omega), \Omega),\,(\chX_s, \chB_s + \ii \chom_s, \chOm_s),
\end{equation*}
with sublattice $H = \langle \sigma_0, E \rangle \subset \bL$, which is admissible in the sense of Definition \ref{AdmissDef}.

The first step is an application of the Gamma property in order to determine the cohomology class of the Lagrangian $\cL_s$ mirror to $\olo_{\cX_{\xi}}$ (see Definition \ref{StructureObject}).
\begin{lemma}\label{GammaLemma} Suppose that homological mirror symmetry for $\chX_s$ holds compatibly with Hodge-theoretic mirror symmetry in the sense of Definition \ref{SYZtypeProperties}. Then for $s \gg 0$ we have
\begin{equation*}
[\cL_s] = \pd(\sigma_0) \in H_2(\chX_s, \Z).
\end{equation*}
\end{lemma}
\begin{proof} Fix $s > 0$. The class of the holomorphic volume form $\chOm_{s}$ of the Hodge-theoretic mirror to $(X,  s(\bB + \ii \omega), \Omega)$ is given by
\begin{align*}
(s E.\Rea\Omega)^{-1}\left(\sigma_0 +  s(\hat{\bB} + \ii \omega)  + \left(\frac{ s^2 (\omega^2 - \hat{\bB}^2)}{2} + 1 - \ii s\omega.(\sigma_0 + s\hat{\bB})\right)E\right). 
\end{align*} 

Using $E.\sigma_0 = 1$, $\hat{\bB}.\sigma_0 = 0$ we compute 
\begin{align*}
& ( E.\Rea\Omega)\chOm_{s}.\sigma_0\\
& = s^{-1}\Big(\sigma_0 +  s(\hat{\bB} + \ii  \omega)  +  \\
& \left(\frac{ s^2( \omega^2 - \hat{\bB}^2 )}{2} + 1 - \ii s\omega.(\sigma_0 +  s\hat{\bB})\right)E\Big)(\sigma_0)\\
&= -\frac{2}{s} + \ii \omega.\sigma_0 + \left(\frac{ s (\omega^2 - \hat{\bB}^2) }{2} + \frac{1}{s} - \ii \omega.(\sigma_0 + s\hat{\bB})\right) \\
&= -\frac{1}{s} + s \left(\frac{\omega^2 - \hat{\bB}^2}{2} - \ii \omega.\hat{\bB}\right) = -\frac{1}{s} + s\frac{(\omega - \ii \hat{\bB})^2}{2}.
\end{align*}
On the other hand, we have
\begin{align*}
\int_X e^{u \omega - \ii u\hat{\bB}} = u^2 \frac{(\omega - \ii \hat{\bB})^2}{2}. 
\end{align*}
Thus the Gamma identity
\begin{equation*}
\int_X e^{u \omega_X - \ii u \hat{\bB}} = \int_{\cL(\olo_X)} \chOm_s + O(s^{-1}),\,s = u^2   
\end{equation*} 
holds with 
\begin{equation*}
[\cL(\olo_X)] = \PD(\sigma_0) \in H_2(\chX_s, \Z)  
\end{equation*}
for all sufficiently large $s > 0$.
\end{proof}
As mentioned in the Introduction, our second step is showing that, under the assumption that $\cL_s$ is \emph{stable}, one can construct a suitable \emph{family} $\cZ$ of K3 surfaces, with a special fibre $\cZ_1$ isomorphic to $(\chX_s, \chB_s + \ii\chom_s, \chOm_s)$, and a further special fibre $\cZ_0$, on which a sLag with required properties exists. A sufficiently strong form of homological mirror symmetry then allows to deform this solution from $\cZ_0$ to $\cZ_1 = \chX$. The upshot of this construction is illustrated in the diagrams \eqref{Diagram1}, \eqref{Diagram2} below. The construction involves a study of hyperk\"ahler rotation combined with a characterisation of stability for $\cL_s$. 
\begin{definition}[\cite{GrossWilson}, Section 1]\label{GrossWilsonRotDef} Fix an angle $\psi$. We define the \emph{Gross-Wilson hyperk\"ahler rotation} of the structure $(\chX_s, \chom_s, \chOm_s)$ with phase $e^{\ii\psi}$, namely 
\begin{equation*}
(\chX^{\psi}_K, \chom^{\psi}_K, \chOm^{ \psi}_K)
\end{equation*}
by the identities
\begin{align*}
\chOm^{ \psi}_{K} &= \Imm(e^{\ii\psi}\chOm_s)+ \ii \chom_s,\\
\chom^{ \psi}_{K} &= \Rea(e^{\ii\psi} \chOm_s) 
\end{align*}  
at the level of forms.
\end{definition}
\begin{rmk} Gross-Wilson only consider the angle $\psi = 0$.
\end{rmk}
Let us write down the rotation explicitly for our applications.
\begin{lemma}\label{HKCohClasses} Introduce the classes
\begin{align*}
&\alpha := \sigma_0 + s\hat{\bB}+\left(\frac{s^2(\omega^2 - \hat{\bB}^2)}{2} + 1 \right)E\\
&\beta := s\omega + \left( - s\omega.(\sigma_0 + s\hat{\bB})\right)E.
\end{align*}
Then, using the specific choice of representative $\hat{\bB}$ introduced in Lemma \ref{GrossWilsonLem}, the holomorphic volume and K\"ahler classes of the Gross-Wilson hyperk\"ahler rotation are given by
\begin{align*}
\chOm^{ \psi}_{K} &= (E.\Rea\Omega)^{-1}\left(\sin(\psi)\alpha + \cos(\psi)\beta\right)\\
& + \ii (E.\Rea\Omega)^{-1}\left(\Imm \Omega - \Imm \Omega . (\sigma_0 + s\hat{\bB}) E \right),\\
\chom^{ \psi}_{K} &= ( E.\Rea\Omega)^{-1}\left(\cos(\psi)\alpha - \sin(\psi)\beta\right).
\end{align*}
\end{lemma}
\begin{cor} Choose the angle $\hat{\psi}$ (modulo $2\pi$) such that
\begin{equation*}
\cos(\hat{\psi}) = \frac{\alpha^2 - \beta^2}{\left((\alpha^2 - \beta^2)^2 + 4 (\alpha.\beta)^2\right)^{\frac{1}{2}}},\,\sin(\hat{\psi}) = -\frac{2\alpha.\beta}{\left((\alpha^2 - \beta^2)^2 + 4 (\alpha.\beta)^2\right)^{\frac{1}{2}}}.
\end{equation*}
Then, introducing an angle $\hat{\theta}$ (modulo $2\pi$), depending on $s$, with
\begin{align*}
\cos(\hat{\theta}) &= \frac{s^2(\omega^2 - \hat{\bB}^2)  + 2}{\left((s^2(\omega^2 - \hat{\bB}^2)  + 2)^2+(2 s\omega.(\sigma_0 + s\hat{\bB}))^2\right)^{\frac{1}{2}}},\\
\sin(\hat{\theta}) &= \frac{2 s\omega.(\sigma_0 + s\hat{\bB}) }{\left((s^2(\omega^2 - \hat{\bB}^2)  + 2)^2+(2 s\omega.(\sigma_0 + s\hat{\bB}))^2\right)^{\frac{1}{2}}},  
\end{align*}
we have
\begin{equation*}
\chom^{ \psi}_{K} = (s E.\Rea\Omega)^{-1} \hat{\eta}, 
\end{equation*}
where
\begin{equation*}
\hat{\eta} := \sin(\hat{\theta})\left(\sigma_0 + s\hat{\bB} \right) + \cos(\hat{\theta}) s \omega
\end{equation*}
is a K\"ahler class on the Gross-Wilson hyperk\"ahler rotation $(\chX^{\hat{\psi}}_K, \chom^{\hat{\psi}}_K, \chOm^{\hat{\psi}}_K)$. In particular, assuming the further normalisation condition $E.\Rea\Omega = 1$ discussed in Remark \ref{NormalisationRmk}, we have
\begin{equation*}
\chom^{ \psi}_{K} = s^{-1}\hat{\eta} = \sin(\hat{\theta})\left(s^{-1}\sigma_0 + \hat{\bB} \right) + \cos(\hat{\theta}) \omega. 
\end{equation*}
\end{cor}
\begin{proof} We compute
\begin{align*}
\alpha^2 &= -2 + s^2\hat{\bB}^2 + 2\left(\frac{s^2(\omega^2 - \hat{\bB}^2)}{2} + 1 \right) = s^2\omega^2,\\
\beta^2 &= s^2\omega^2 - 2 s\omega.(\sigma_0 + s\hat{\bB})(E. s\omega),
\end{align*}
so we have
\begin{align*}
\alpha^2 -\beta^2&= 2 s\omega.(\sigma_0 + s\hat{\bB})(E.s\omega). 
\end{align*}
Similarly
\begin{align*}
 \alpha.\beta &= \sigma_0. s\omega - s\omega.(\sigma_0 + s\hat{\bB}) + s^2\hat{\bB}.\omega + \left(\frac{s^2(\omega^2 - \hat{\bB}^2)}{2} + 1 \right)(E.s\omega)\\
&= \left(\frac{s^2(\omega^2 - \hat{\bB}^2)}{2} + 1 \right)(E.s\omega).
\end{align*}
By construction, 
\begin{equation*}
\chom^{\psi}_K = ( E.\Rea\Omega)^{-1}\left(\cos(\psi)\alpha - \sin(\psi)\beta\right) 
\end{equation*}
is a K\"ahler class on $(\chX^{\psi}_K, \chom^{\psi}_K, \chOm^{ \psi}_K)$ for all $\psi$. Specialising to $\hat{\psi}$ we see that the class 
\begin{align*}
\hat{\eta}:=\frac{(\alpha^2 -\beta^2) \alpha + (2 \alpha.\beta) \beta}{\left((\alpha^2 - \beta^2)^2 + 4 (\alpha.\beta)^2\right)^{\frac{1}{2}}} 
\end{align*}
is a K\"ahler class on $(\chX^{\hat{\psi}}_K, \chom^{\hat{\psi}}_K, \chOm^{\hat{\psi}}_K)$. Direct computation shows 
\begin{align*}
&(\alpha^2 -\beta^2) \alpha + (2 \alpha.\beta) \beta \\&= 2 s\omega.(\sigma_0 + s\hat{\bB}) \left(\sigma_0 + s\hat{\bB}+\left(\frac{s^2(\omega^2 - \hat{\bB}^2)}{2} + 1 \right)E\right)(E.s\omega)\\
& + \left(s^2(\omega^2 - \hat{\bB}^2)  + 2 \right) \left(s\omega + \left( - s\omega.(\sigma_0 + s\hat{\bB})\right)E\right)(E.s\omega)\\&= \left(2 s\omega.(\sigma_0 + s\hat{\bB}) \left(\sigma_0 + s\hat{\bB} \right) + \left(s^2(\omega^2 - \hat{\bB}^2)  + 2 \right) \omega\right)(E.s\omega),
\end{align*}
as well as
\begin{align*}
&\left((\alpha^2 - \beta^2)^2 + 4 (\alpha.\beta)^2\right)^{\frac{1}{2}} \\
&= \left((s^2(\omega^2 - \hat{\bB}^2)  + 2)^2+(2 s\omega.(\sigma_0 + s\hat{\bB}))^2\right)^{\frac{1}{2}}(E.s\omega).
\end{align*}
The claim follows. 
\end{proof} 
 
\begin{definition} We say that the K3 surface $(X, s(\bB + \ii \omega), \Omega)$ with sublattice $H = \langle \sigma_0, E \rangle \subset \bL$ satisfies the \emph{twisted ampleness condition} if 
\begin{enumerate}
\item[$(i)$] $\sigma_0 + s\hat{\bB}$ is a real $(1,1)$-class on $(X, \Omega)$;
\item[$(ii)$] denoting by $\hat{\theta}$ the unique angle (modulo $2\pi$) such that  
\begin{equation*}
\Imm e^{-\ii \hat{\theta}} (s\omega + \ii (\sigma_0 + s\hat{\bB}))^2 = 0,
\end{equation*}
in cohomology, we have that 
\begin{equation*}
\Rea e^{-\ii \hat{\theta}} (s\omega + \ii (\sigma_0 + s\hat{\bB})) 
\end{equation*}
is a K\"ahler class on $(X, \Omega)$.
\end{enumerate}
\end{definition} 
A straightforward computation then gives the following characterisation.
\begin{lemma}\label{TwistedPosLem} Suppose that $\sigma_0 + s\hat{\bB}$ is a real $(1,1)$-class on $(X, \Omega)$. Then the angle $\hat{\theta}$ (modulo $2\pi$) is given by 
\begin{align*}
\cos(\hat{\theta}) &= \frac{s^2(\omega^2 - \hat{\bB}^2)  + 2}{\left((s^2(\omega^2 - \hat{\bB}^2)  + 2)^2+(2 s\omega.(\sigma_0 + s\hat{\bB}))^2\right)^{\frac{1}{2}}},\\
\sin(\hat{\theta}) &= \frac{2 s\omega.(\sigma_0 + s\hat{\bB}) }{\left((s^2(\omega^2 - \hat{\bB}^2)  + 2)^2+(2 s\omega.(\sigma_0 + s\hat{\bB}))^2\right)^{\frac{1}{2}}},  
\end{align*}
and the twisted ampleness condition on $(X, s(\bB + \ii \omega), \Omega)$ with sublattice $H = \langle \sigma_0, E \rangle \subset \bL$ is equivalent to the condition that the class
\begin{equation*}
\hat{\eta} = \sin(\hat{\theta})\left(\sigma_0 + s\hat{\bB} \right) + \cos(\hat{\theta}) s \omega
\end{equation*}
is K\"ahler on $(X, \Omega)$.
\end{lemma}
We note an immediate implication.
\begin{cor} Suppose that $(X, s(\bB + \ii \omega), \Omega)$, with sublattice $H = \langle \sigma_0, E \rangle \subset \bL$, satisfies the twisted ampleness condition. Set
\begin{equation*}
\eta := (s E.\Rea\Omega)^{-1} \hat{\eta} = (s E.\Rea\Omega)^{-1} \left(\sin(\hat{\theta})\left(\sigma_0 + s\hat{\bB} \right) + \cos(\hat{\theta}) s\omega\right)
\end{equation*}
In particular, assuming the further normalisation condition $E.\Rea\Omega = 1$ of Remark \ref{NormalisationRmk}, we have
\begin{equation*}
\eta = s^{-1}\hat{\eta}. 
\end{equation*}
Then 
\begin{equation*}
(X, \hat{\eta}, \Omega),\, (\chX^{\hat{\psi}}_K, \hat{\eta}, \chOm^{\hat{\psi}}_K)
\end{equation*}
are both K3 surfaces endowed with a fixed K\"ahler class $\hat{\eta} = s\eta$. 
\end{cor}

\begin{cor}\label{TwistedToFamilyCor} Suppose that $(X, s(\bB + \ii \omega), \Omega)$, with sublattice $H = \langle \sigma_0, E \rangle \subset \bL$, satisfies the twisted ampleness condition. There exists a holomorphic family of K3 surfaces endowed with K\"ahler classes
\begin{equation*}
\cZ := (\cZ_t, \omega_{\cZ_t}, \Omega_{\cZ_t}) \to T
\end{equation*} 
over an analytic space $T$, with special fibres $(\cZ_i, \omega_{\cZ_i}, \Omega_{\cZ_i})$, $i=0,1$, satisfying the following properties:
\begin{enumerate}
\item[$(i)$]  $(\cZ_1, \omega_{\cZ_1}, \Omega_{\cZ_1})$ isomorphic to $(\chX_s, \chom_s, \chOm_s)$;
\item[$(ii)$] $(X, \eta, \Omega)$ is isomorphic to the hyperk\"ahler rotation of $(\cZ_0, \omega_{\cZ_0}, \Omega_{\cZ_0})$ with angle $\hat{\psi}$;
\item[$(iii)$] assuming the normalisation condition $E.\Rea\Omega = 1$ of Remark \ref{NormalisationRmk}, the cohomology class $\omega_{\cZ_t}$ is constant along the family.   
\end{enumerate}
\end{cor}
\begin{proof}
By the work of Schumacher \cite{Schumacher_CY} there exists a coarse moduli space of K3 surfaces endowed with a K\"ahler class $\eta$ given by a complex analytic space $\mathscr{M}(\eta)$. Thus, there exists a family 
\begin{equation*}
\cY \to T
\end{equation*}
of K3 surfaces $(\cY_t, \hat{\eta}, \Omega_{\cY_t})$ over a base analytic space $T$, corresponding to an analytic map $T \to \mathscr{M}(\eta)$, such that the fibres $\cY_{0} := \cY_{t_0}$, $\cY_{1} := \cY_{t_1}$ over points $t_0, t_1 \in T$ are given by
\begin{equation*}
\cY_0 \cong (X, \hat{\eta}, \Omega),\,\cY_1 \cong (\chX^{\hat{\psi}}_K, \hat{\eta}, \chOm^{\hat{\psi}}_K).
\end{equation*} 

We denote by 
\begin{equation*}
\cZ \to T
\end{equation*} 
the unique family of K3 surfaces such that each fibre $(\cY_{t}, \eta = s^{-1}\hat{\eta}, \Omega_{\cY_t})$ is obtained as the hyperk\"ahler rotation with angle $\hat{\psi}$ of $(\cZ_t, \omega_{\cZ_t}, \Omega_{\cZ_t})$. We write $(\cZ_i, \omega_{\cZ_i}, \Omega_{\cZ_i}) := (\cZ_{t_i}, \omega_{\cZ_{t_i}}, \Omega_{\cZ_{t_i}})$ for $i =0, 1$. 

Then, by construction, we have an isomorphism of K3 surfaces
\begin{equation*}
(\chX_s, \chom_s, \chOm_s) \cong (\cZ_1, \omega_{\cZ_1}, \Omega_{\cZ_1}),
\end{equation*} 
while $(X, \eta, \Omega)$ is obtained as the hyperk\"ahler rotation of $(\cZ_0, \omega_{\cZ_0}, \Omega_{\cZ_0})$ with angle $\hat{\psi}$.

As for the claim $(iii)$, we note that, by construction, we have $\Omega_{\cY_0} = \Omega$ and, Lemma \ref{HKCohClasses},
\begin{align*}
\Omega_{\cY_1} = \chOm^{ \hat{\psi}}_{K} &= (E.\Rea\Omega)^{-1}\left(\sin(\hat{\psi})\alpha + \cos(\hat{\psi})\beta\right)\\
& + \ii (E.\Rea\Omega)^{-1}\left(\Imm \Omega - \Imm \Omega . (\sigma_0 + s\hat{\bB}) E \right).
\end{align*} 
But, under our assumptions, we have
\begin{equation*}
\Omega . (\sigma_0 + s\hat{\bB}) = 0,\,E.\Rea\Omega = 1,
\end{equation*}
so we find 
\begin{align*}
\Imm \Omega_{\cY_0} = \Imm \Omega_{\cY_1} = \Imm \Omega. 
\end{align*} 
By the K\"ahler Torelli theorem (\cite{HuyK3Book}, Theorem 7.5.3), there is a variation of complex structure for $(X, \hat{\eta}, \Omega)$, compatible with $\hat{\eta}$, uniquely determined by a family of complex forms $\Omega'_t$, satisfying
\begin{align*}
&\Imm \Omega'_t = \Imm \Omega,\,t\in T,\\
&\Omega'_0 = \Omega,\,\Omega'_1 = \Omega_{\cY_1} = \chOm^{ \hat{\psi}}_{K}. 
\end{align*}  
So, up to replacing $\cY$ with the family determined by $\Omega'_t$, we can assume that
\begin{align*}
\Imm \Omega_{\cY_t} = \Imm \Omega,\,t \in T. 
\end{align*} 
By the definition of the family $\cZ$ in terms of hyperk\"ahler rotation, we have
\begin{align*}
\Omega_{\cY_t} &= \Imm(e^{\ii\hat{\psi}}\Omega_{\cZ_t})+ \ii \omega_{\cZ_t},
\end{align*}  
so $\omega_{\cZ_t} = \Imm \Omega_{\cY_t} = \Imm \Omega$ is constant as claimed. 
\end{proof}
We may now introduce the relevant Bridgeland stability conditions.  
\begin{definition} Let $S$ be a projective surface over an algebraically closed field. For $s > 0$, $\omega_S, \beta$ in $\NS(S)_{\R}$ with $\omega_S$ ample, we let $\tau_s$ be the canonical Bridgeland stability condition on $\DCoh(X)$ with central charge
\begin{equation*}
\opZ_s(F) = -\int_S e^{- (\ii s \omega_S + \beta)} \ch(F),
\end{equation*}   
with heart $\cA$ obtained by a single tilt of the standard heart $\Coh(S)$ (see \cite{BridgelandK3}).
\end{definition}
\begin{rmk} In our applications the algebraically closed base field is given by the universal Novikov field $\Lambda$ with the canonical inclusion $\C\pow{q} \subset \Lambda$, see \cite{SheridanSmithGP}, Section 1.4.
\end{rmk}
The following result, initially proved by the author and extended by Fan, provides the key link between Bridgeland stability and twisted ampleness. 
\begin{thm}[\cite{J_toricThomasYau}, Section 6 and \cite{YWFan_stability}, Theorem 1.3]\label{TwistedPosThm} Fix $S$ a projective surface over an algebraically closed field. Let $\cE$ be a line bundle on $S$. Suppose $\cE^s := \cE^{\otimes s}$ is $\tau_s$-stable for $s \gg 0$ and $\omega_S$ is general. Then the following twisted ampleness holds: denoting by $\bar{\theta}$ the unique angle (modulo $2\pi$) such that  
\begin{equation*}
\Imm e^{-\ii \bar{\theta}} (s\omega + \ii (s c_1(\cE) + \beta))^2 = 0,
\end{equation*}
in cohomology, we have that 
\begin{equation*}
\Rea e^{-\ii \bar{\theta}} (s\omega + \ii (s c_1(\cE) + \beta)) 
\end{equation*}
lies in the ample cone $\Amp(S) \subset \NS(S)_{\R}$.
\end{thm}
We will apply it in the following form.
\begin{cor}\label{TwistedPosCor}  Suppose 
\begin{equation*}
\psi^*\DCoh(\cX_{\xi}) \cong \DFuk(\chX_s, \chB_s + \ii \chom_s), 
\end{equation*}
satisfies the properties of Definition \ref{SYZtypeProperties}. Let $\sigma_s$ denote the standard Bridgeland stability condition on $\DCoh(\cX_{\xi})$ with central charge
\begin{equation*}
\opZ_s(F) = -\int_X e^{- (\ii s\omega + s\hat{\bB} + \sigma_0)} \ch(F),
\end{equation*}  
where we use the identification $i\!: \NS(X)_{\R} \cong \NS(\cX_{\xi})_{\R}$ induced by $i\!:\Pic(X) \cong \Pic(\cX_{\xi})$ appearing in Definition \ref{SYZtypeProperties}, $(i)$. Assume that $\olo_{\cX_{\xi}} \in \Pic(X)$ is $\sigma_s$-stable for $s \gg 0$ and $\omega$ is general (i.e. does not lie on the union of countably many analytic subvarieties in $H^{1,1}(X, \R)$). Then twisted ampleness holds on $(X, s(\bB + \ii \omega), \Omega)$ with sublattice $H = \langle \sigma_0, E \rangle \subset \bL$: the class
\begin{equation*}
\hat{\eta} = \sin(\hat{\theta})\left(\sigma_0 + s\hat{\bB} \right) + \cos(\hat{\theta}) s\omega
\end{equation*}
is K\"ahler on $(X, \Omega)$.  
\end{cor}
\begin{proof} This follows at once from Theorem \ref{TwistedPosThm} and our assumption in Definition \ref{SYZtypeProperties}, $(ii)$ that $i\!: \NS(X)_{\R} \cong \NS(\cX_{\xi})_{\R}$ preserves the ample cones. 
\end{proof}
\begin{definition} For $s > 0$, we let $\chs_s$ be the Bridgeland stability condition on $\DFuk(\chX_s, \chom_s)$ induced from $\sigma_s = (\cA, \opZ_s)$ through the equivalence
\begin{equation*}
\DFuk(\chX_s, \chB_s+ \ii \chom_s) \cong \psi^*\DCoh(\cX_{\xi}).
\end{equation*}
\end{definition}
\begin{cor}\label{StabToFamilyCor} Suppose $\omega$ is general and $\cL_s \in \DFuk(\chX_s, \chB_s+\ii\chom_s)$ is $\chs_s$-stable for $s \gg 0$. Then there exists a family $\cZ \to T$ satisfying the properties of Corollary \ref{TwistedToFamilyCor}.  
\end{cor}
\begin{proof} The claim follows immediately from the Corollaries \ref{TwistedToFamilyCor} and \ref{TwistedPosCor}.
\end{proof}
The following diagrams illustrate our construction so far. Initially, we only have a triangle
\begin{equation}\label{Diagram1}
  \begin{tikzcd}
      & (\chX_s, \chB_s + \ii \chom_s, \chOm_s) \arrow{d}{\hat{\psi}-\textrm{rotation}}\\
    (X, s(\bB + \ii \omega), \Omega)\ar{ur}{\textrm{Hodge-theoretic mirror map}} & (\chX^{\hat{\psi}}_K, \hat{\eta}, \chOm^{\hat{\psi}}_K)
  \end{tikzcd}
\end{equation}
but, under the condition that $\cL_s \in \DFuk(\chX_s, \chB_s+\ii\chom_s)$ is $\chs_s$-stable for $s \gg 0$, we can construct the auxiliary families
\begin{equation}\label{Diagram2}
  \begin{tikzcd}
   &  (\cZ_t, \omega_{\cZ_t}, \Omega_{\cZ_t})\ar{dl}\ar{dr}\ar{dddd}{\hat{\psi}-\textrm{rotation}}&\\
      (\cZ_0, \omega_{\cZ_0}, \Omega_{\cZ_0}) \arrow{d}{\hat{\psi}-\textrm{rotation}}  & &  (\chX_s, \chom_s, \chOm_s) \arrow{d}{\hat{\psi}-\textrm{rotation}} \\
    (X, \hat{\eta}, \Omega)\arrow{d}{\cong} & & (\chX^{\hat{\psi}}_K, \hat{\eta}, \chOm^{\hat{\psi}}_K)\arrow{d}{\cong} \\
    (\cY_0, \omega_{\cY_0}, \Omega_{\cY_0})  & & (\cY_1, \omega_{\cY_1}, \Omega_{\cY_1})\\
    & (\cY_t, \omega_{\cY_t}, \Omega_{\cY_t})\ar{ul}\ar{ur} &  
  \end{tikzcd}
\end{equation}
satisfying the properties of Corollary \ref{TwistedToFamilyCor}. In particular we have $\omega_{\cY_t} \equiv \hat{\eta}$ and, assuming the normalisation condition $E.\Rea\Omega = 1$ of Remark \ref{NormalisationRmk}, the cohomology class $\omega_{\cZ_t}$ is constant.
\begin{lemma}\label{SmoothnessLem} Assume $\omega$ is general and $\cL_s \in \DFuk(\chX_s, \chB_s+\ii\chom_s)$ is $\chs_s$-stable for $s \gg 0$, so we have our family of K3 surfaces $\cZ \to T$. Then there exists a closed $0$-dimensional analytic subset $T_{\rm sing} \subset T$, with $t_1 \notin T_{\rm sing}$, such that, for all $t \in T\setminus T_0$, the class $\PD(\sigma_0) \in H_2(\cZ_t, \Z)$ admits a unique smooth special Lagrangian representative $L_t \subset (\cZ_t, \omega_{\cZ_t}, \Omega_{\cZ_t})$, homeomorphic to a sphere. The family $L_t$ is $C^{\infty}$.
\end{lemma}
\begin{proof} The claim that $\PD(\sigma_0) \in H_2(\chX, \Z)$ admits a unique smooth special Lagrangian representative is standard, following from $\sigma^2_0 = -2$ and a sufficiently general choice of $\omega$, yielding a general choice of $(\chX, \chB + \ii \chom, \chOm)$, see e.g. \cite{YLShen_K3FibredSLag}, Lemma 2.3. The proof of the same Lemma actually shows that the condition that such a smooth sLag representative (which is then necessarily unique) exists holds away from a closed analytic subset on the base of a holomorphic submersion. Smooth dependence on parameters follows from ellipticity properties of the special Lagrangian equation.
\end{proof}
\begin{proof}[Completion of the proof of Theorem \ref{MainThm}] Suppose $\cL_s \in \DFuk(\chX_s, \chom_s)$ is $\chs_s$-stable for $s \gg 0$ with respect to a general $\omega$, so we have our family of K3 surfaces $\cZ \to T$ with special fibres 
\begin{equation*}
(\cZ_0, \omega_{\cZ_0}, \Omega_{\cZ_0}),\,(\cZ_1, \omega_{\cZ_1}, \Omega_{\cZ_1}) \cong (\chX, \chom, \chOm),
\end{equation*} 
constructed in Corollary \ref{TwistedToFamilyCor}. We endow them with the same B-field class $\chB_s$. By $(iii)$ of Corollary \ref{TwistedToFamilyCor}, we know that, up to fixing the normalisation $E.\Rea\Omega = 1$ (which is not restrictive, as explained in Remark \ref{NormalisationRmk}), the family $\cZ \to T$ is in fact trivial symplectically, i.e. the cohomology class of $\omega_{\cZ_t}$ is constant. So we have a family of fibrewise homological SYZ fibrations
\begin{equation*}
\pi_t\!: \cZ_t \to S^2,
\end{equation*}
with constant fibre class $E$, and the fibrewise categories $\DFuk(\cZ_t, \chB_s+\ii\chom_s)$ fit together in a local system of categories $\DFuk(\cZ / T)$.

Note that, by the classical mirror relations \eqref{ClassicalMirrorMap} and Lemma \ref{GrossWilsonLem}, we have 
\begin{align*}
&\chB_s = (E. \Rea \Omega)^{-1} \Rea \Omega - \sigma_0\mod E \\
&= \Rea \Omega - \sigma_0\mod E,\\
&\chom_s = (E.\Rea\Omega)^{-1}\left(\Imm \Omega - (\Imm \Omega . (\sigma_0 + s\hat{\bB})) E \right)\\& = \Imm \Omega,  
\end{align*}
by our assumptions and the normalisation $E.\Rea\Omega = 1$. The Gross-Wilson representative of $\chB_s$, for which $\chB_s.\sigma_0 = 0$, is given by
\begin{align*}
\chB_s = \Rea \Omega - \sigma_0 - 2E.
\end{align*}

So we see that, if $\Rea \Omega \in H^{2}(X, \Z)$, then the B-field classes $\chB_s$ are integral, and in particular we have equivalences 
\begin{equation*}
\DFuk(\chX_s, \chB_s + \ii \chom_s) \cong \DFuk(\chX_s, \chom_s) 
\end{equation*}
(see \cite{Sheridan_versality}, Remark 4.11).

Now we replace $T$ with the base $T\setminus T_0$ constructed in Lemma \ref{SmoothnessLem}. Then there is a section $\LL \in \DFuk(\cZ / T)$ of our local system, restricting on each fibre to the object represented by the unique sLag sphere in the class $\PD(\sigma_0) \in H_2(\cZ_t, \Z)$, i.e. such that
\begin{equation*}
\LL_t := \LL|_{\cZ_t} \cong L_t \in \DFuk(\cZ_t, \chB_{\cZ_t} + \ii\omega_{\cZ_t}),
\end{equation*}
endowed with \emph{constant} B-field class $\chB_{\cZ_t} := \chB_s$.

We consider the relative mirror family $\chcalZ \to T$ over $T$ of K3 surfaces defined over $\C\pow{q}$ satisfying 
\begin{equation*}
\DFuk(\cZ_t) \cong \psi^*_t\DCoh((\chcalZ_{t})_{\xi}).
\end{equation*}
This specialises to equivalences
\begin{align*}
\DFuk(\cZ_0, \chB_{\cZ_0} + \ii\omega_{\cZ_0}) &\cong \psi^*_0\DCoh((\chcalZ_0)_{\xi}),\\
\DFuk(\cZ_1, \chB_{\cZ_1} + \ii\omega_{\cZ_1}) &\cong \psi^*\DCoh(\cX_{\xi}) \cong \psi^*_1\DCoh((\chcalZ_1)_{\xi}).
\end{align*} 
We are assuming that $\sigma_0$ is type $(1,1)$ and that $\Omega$ is sufficiently general, so $\sigma_0$ can be represented by a smooth divisor $D \subset X$. The homological SYZ fibration $\pi_0\!: \cZ_0 \to S^2$ has fibre class $E$, satisfying $E.\sigma_0 = 1$, so if we regard $D \subset \cZ_0$ as a special Lagrangian submanifold, using standard properties of hyperk\"ahler rotation (see e.g. \cite{YLShen_K3FibredSLag}, Lemma 2.3), then $D$ is a sLag section of $\pi_0$ in the cohomology class $\pd(\sigma_0)$ and by uniqueness we have $D = L_0$. It follows that there is a line bundle $\cE_0 \in \Pic(\chcalZ_0)$ such that the object $\LL_0 \in \DFuk(\cZ_0, \omega_{\cZ_0})$ is mirror to $\cE_0$.

Let $\cF$ denote the locus of points $t\in T$ for which the object  $\cE_t$ mirror to $\LL_t \in \DFuk(\cZ_t)$ is given by a line bundle. Then $\cF$ is non-empty since we showed $t_0 \in \cF$. 
  
We claim that $\cF \subset T$ is Zariski open and closed. By the equivalence $\DFuk(\cZ / T) \cong \DCoh(\chcalZ / T)$ this becomes a claim about the family $\cE_t \in \DCoh(\chcalZ / T)$ mirror to $\LL_t$. This is a flat family, since at each point it is mirror to a deformation of $\LL_t$ as an object of $\DFuk(\cZ_t)$. The latter claim follows from the general deformation theory of sLags under a variation of the holomorphic volume form, as discussed by Joyce in \cite{Joyce_ConicSLagSurvey}, Section 2.4, i.e. these results show that deforming $L_t$ by varying $\Omega_{\cZ_t}$ induces a deformation of $\LL_t$ as an object of $\DFuk(\cZ_t)$. On the other hand, since $L_t$ is a Lagrangian sphere, all these deformation classes in $\DFuk(\cZ_t)$ are trivial by McLean's theorem, so indeed $\cF \subset T$ must be open and closed.

We elucidate this argument further from a slightly different point of view. Deformations of line bundles on a fibre $\chcalZ_t$ as objects in the derived category remain line bundles (see e.g. \cite{PandhaThomas}, Section 2 for more general results). Thus $\cF$ is open in the Zariski topology.

It follows that $\cE_1 \in \DCoh(\chcalZ_1) \cong \DCoh(\cX_{\xi})$ is in fact the flat limit of a flat family of line bundles on the fibres of $\chcalZ_{\cF} \to \cF$. As such it is a rank $1$ coherent sheaf $F$ on the surface $\cX_{\xi}$. If $F$ is torsion free but not locally free, then it is isomorphic to an object $\cI_S \otimes F'$ where $F'$ is a line bundle and $S \subset \cX_{\xi}$ is a non-empty zero-dimensional subscheme. But this has the wrong topology for a flat limit. Thus, if it is not locally free, $F$ must have torsion, and in an open set around $t_1$ it must smooth to a line bundle. This implies that $F$ has nontrivial deformations as on object of $\Coh(\cX_{\xi})$, contradicting the rigidity of its mirror sLag sphere $L_1$ given by McLean's theorem. 

This shows that $\cE_1$ is given by a line bundle, which is then necessarily isomorphic to $\olo_{\cX_{\xi}}$ by Lemma \ref{GammaLemma}.  

Finally, suppose that $X$, $\chX_s$ is a pair of Green-Plesser type. We observe that, under the additional assumptions:
\begin{enumerate}
\item[$(a)$] $\hat{\bB}$ is induced by a class in $\NS(\widehat{\PP})_{\R}$, where $\widehat{\PP}$ denotes $\PP^3$ or a crepant resolution of $\PP(3,1,1,1)$,
\item[$(b)$] the Gross-Wilson normalised holomorphic volume form satisfies the further conditions that $\Rea \Omega$ is integral and $\chom_s = \Imm \Omega$ is also the pullback of a real Neron-Severi class from the ambient (our general choice of $\Omega$ can be made compatibly with this properties),
\end{enumerate}
the results summarised in Proposition \ref{MirrorProp} all apply, in particular since $\chB_s$ is integral and we have equivalences
\begin{equation*}
\DFuk(\cZ_t, \chB_{\cZ_t} + \ii\omega_{\cZ_t}) \cong \DFuk(\cZ_t, \omega_{\cZ_t}). 
\end{equation*}
\end{proof}
\section{Examples on K3 $\times C$}\label{ProductSec}
 For $s > 0$, let us fix a pair 
\begin{equation*}
(X, s(\bB + \ii \omega), \Omega),\,(\chX_s, \chB_s + \ii \chom_s, \chOm_s)
\end{equation*}
as in the previous sections. We also fix mirror elliptic curves 
\begin{equation*}
(C, \omega_C, \Omega_C),\,(\chC, \omega_{\chC}, \Omega_{\chC}),
\end{equation*}
satisfying 
\begin{equation*}
\DCoh(C) \cong \DFuk(\chC, \omega_{\chC})
\end{equation*}
(see the classical work of Polishchuck and Zaslow \cite{PolishZaslow}). 

Consider the product Calabi-Yau threefolds, \emph{depending on $s > 0$}, given by
\begin{align*}
(\cZ_{\xi}, s\bB+\ii\omega_{Z,s}, \Omega_Z)  &:= (\cX_{\xi} \times C, s \bB + \ii ( s p^*\omega + q^*\omega_C), p^*\Omega \wedge q^* \Omega_C)),\\
(\chZ_s, \chB_s + \ii\omega_{\chZ_s}, \Omega_{\chZ}) &:= (\chX_s \times \chC, \chB_s + \ii(p^* \chom_s + q^*\omega_{\chC}), p^* \chOm_s \wedge q^*\Omega_{\chC}),
\end{align*}
where we denote the projections (using slightly overcharged notation) by
\begin{equation*}
p\!: \cZ_{\xi} \to \cX_{\xi},\,q\!: \cZ_{\xi} \to C;\,p\!: \chZ_s \to \chX_s,\,q\!: \chZ_s \to \chC.
\end{equation*}
\begin{definition} We say that homological mirror symmetry for $\chZ_s$ holds compatibly with Hodge-theoretic mirror symmetry if this holds for 
\begin{equation*}
\psi^*\DCoh(\cX_{\xi}) \cong \DFuk(\chX_s, \chB_s + \ii \chom_s) 
\end{equation*}
and there is a mirror equivalence
\begin{equation*}
\psi^*\DCoh(\cX_{\xi} \times C) \cong \DFuk(\chZ_s, \omega_{\chZ_s}),
 \end{equation*}
satisfying the following compatibility property: if $L_{\chX} \subset \chX_s$ is a Lagrangian section of a homological SYZ fibration $\pi\!: \chX_s \to S^2$ corresponding to a line bundle $\cE \in \Pic(\cX_{\xi})$, and $L_{\chC} \subset \chC$ is a Lagrangian mirror to a line bundle $\cE_{C} \in \Pic(C)$, then the product $L_{\chX} \times L_{\chC} \subset \chZ_s$ defines an object in $\DFuk(\chZ_s, \omega_{\chZ_s})$ which is mirror to $(p^*\cE) \otimes (q^*\cE_{C})$.
\end{definition}

\begin{definition} Fix any line bundle $\cE_C \in \Pic(C)$. For $s \geq 1$, we denote by  
\begin{equation*}
\cL_{\chZ_s} \in \DFuk(\chZ_s, \omega_{\chZ_s})
\end{equation*} 
the object mirror to 
\begin{equation*}
p^*(\olo_{\cX_{\xi}}) \otimes q^*(\cE_C) \in \Pic(\cZ_{\xi}) 
\end{equation*}
(note that $\cL_{\chZ_s}$ depends on a choice of $\cE_C$).
\end{definition}

Liu \cite{Liu_products} showed how to lift stability conditions from a general variety to its product with any curve. We will only use a weak form of the result, which we summarise as follows.

Suppose $\cA$ is the Noetherian heart of a stability condition $\tau = (\cA, \opZ)$ on $\DCoh(\cX_{\xi})$, such that the central charge $\opZ$ has discrete image. Choose an ample line bundle $\olo(1)$ on $C$. Define a subcategory 
\begin{equation*}
\cA_C := \{F \in \DCoh(\cZ_{\xi})\!: p_*(F \otimes q^*\olo(n)) \in \cA \textrm{ for } n \gg 0\} \subset \DCoh(\cZ_{\xi}).
\end{equation*}
By results of Abramovich and Polishchuck \cite{AbraPolish}, the category $\cA_C$ is in fact the Noetherian heart of a bounded t-structure. 

According to the proof of Theorem 3.3 in \cite{Liu_products}, the function
\begin{equation*}
L_{F}(n) := \opZ(p_*(F \otimes q^*\olo(n))) 
\end{equation*}
is linear for $n \gg 0$.
\begin{thm}[\cite{Liu_products}, Theorem 3.3]\label{LiuWeakThm} For $F \in \cA_C$, define
\begin{equation*}
\opZ_C(F) := \lim_{n \to \infty} \frac{\opZ(p_*(p^*\cE_X \otimes q^*(\olo(n))))}{n \vol(\olo(1))}. 
\end{equation*}
Then $(\cA_{C}, \opZ_C)$ is a weak pre-stability condition on $\DCoh(\cZ_{\xi})$.
\end{thm}
As in the previous Section, let $\sigma_s$ be the standard Bridgeland stability condition on $\DCoh(X)$ with central charge
\begin{equation*}
\opZ_s(F) = -\int_X e^{- (\ii s\omega + s\hat{\bB} +\sigma_0)} \ch(F).
\end{equation*}
\begin{definition} Let $\sigma_{C, s}$ be the weak pre-stability condition on $\DCoh(\cZ_{\xi})$ given by the canonical lift of $\sigma_s$ provided by Theorem \ref{LiuWeakThm}. 
\end{definition}
\begin{definition} We denote by $\chs_{C, s}$ the weak pre-stability condition on $\DFuk(\chZ_s, \omega_{\chZ_s})$ induced by $\sigma_{C, s}$ through the equivalence
\begin{equation*}
\psi^*\DCoh(\cZ_{\xi}) \cong \DFuk(\chZ_s, \omega_{\chZ_s}).
\end{equation*}
\end{definition}
\begin{thm}\label{ProductThm} Suppose that homological mirror symmetry for $\chZ_s$ holds compatibly with Hodge-theoretic mirror symmetry. If $\cL_{\chZ_s}$ is $\chs_{C, s}$-stable for $s \gg 0$, then $\cL_{\chZ_s}$ is isomorphic to a smooth special Lagrangian. 
\end{thm}
\begin{proof} As above, suppose $\cA$ is the Noetherian heart of a stability condition $\tau$ on $\DCoh(\cX_{\xi})$, such that the central charge $\opZ$ has discrete image. We denote by $\tau_C$ the corresponding weak pre-stability condition on $\DCoh(\cZ_{\xi})$. As recalled in \cite{Liu_products}, Lemma 3.8, there is a slicing $\cP_{C}$ on $\DCoh(\cZ_{\xi})$ such that $\cA_C = \cP_C(> \phi) \cap \cP_C(\leq \phi)$. According to \cite{Liu_products}, Proposition 3.14, if $\cF \in \cA_C$ is a $\tau_C$-semistable object of phase $\phi$ with $\opZ_C(\cF) \neq 0$, then there is an exact sequence in $\cA_C$,
\begin{equation*}
0 \to K \to F \to Q \to 0, 
\end{equation*} 
such that $K \in \cP_C(\phi)$, $Q \in \cP_C(< \phi)$ and $\opZ_C(Q) = 0$ (possibly $Q = 0$). Moreover, combining Propositions 3.10 and 3.13 in \cite{Liu_products} we see that if $K$ is t-flat (i.e. for all closed points $c \in C$ we have $K_c \in \cA$) then $K_c$ is $\tau$-semistable of phase $\phi$ for all $c \in C$. 

Suppose now that $F = p^*\cF \otimes q^*\cE_C$ for $\cF \in \cA$ and a line bundle $\cE_C$ on $C$. Then, according to \cite{Liu_products}, Corollary 3.5 $(a)$, we have $\opZ_C(F) = \opZ(\cF)$. Moreover, as recalled in \cite{Liu_products}, Section 3, for any closed embedding $\iota_{C'}\!: C' \hookrightarrow C$, the functor $\iota^*_{C'}\!: \DCoh(X\times C) \to \DCoh(X\times C')$ is t-right exact for the Abramovich-Polishcuck t-structure, so in particular if $F$ is t-flat then we get a surjection in $\cA$,
\begin{equation*}
 F_c \to Q_c \to 0. 
\end{equation*}    
  
We apply these general results to the stability condition $\tau = \sigma_s$ and the t-flat object $\cF = p^*\olo_{\cX_{\xi}} \otimes q^*\cE_C$. Thus, we have $\cZ_C(\cF) = \opZ(\olo_X) \neq 0$, and if $\cF$ is $\sigma_{C,s}$-stable of phase $\phi$, then for all $c \in C$ we have a surjection in $\cA$,
\begin{equation*}
\cF_c \cong \olo_X \to Q_c \to 0, 
\end{equation*}     
where the object $Q$ satisfies $\opZ_C(Q) = 0$. However, according to \cite{Liu_products}, Corollary 3.5 $(b)$, the latter vanishing can only hold if the set $\{c \in C\!: Q_c = 0\}$ is Zariski open. Thus, base-changing to this Zariski open set, we find that $K \cong \cF$ there, so at a generic point $\cF_c$ is $\sigma_s$-semistable of phase $\phi$. But in our case the restrictions $\cF_c \cong \olo_{\cX_{\xi}}$ are isomorphic for all $c$, and our genericity assumptions rule out that $\olo_{\cX_{\xi}}$ is strictly semistable. 

The upshot of our argument is that if $p^*\olo_{\cX_{\xi}} \otimes q^*\cE_C$ is $\sigma_{C, s}$-stable, then $\olo_{\cX_{\xi}}$ is $\sigma_s$-stable. Using our equivalence $\psi^*\DCoh(\cZ_{\xi}) \cong \DFuk(\chZ_s, \omega_{\chZ_s})$ it follows that, if the object $\cL_{\chZ_s} \in \DFuk(\chZ_s, \omega_{\chZ_s})$ is $\chs_{C, s}$-stable, then $\cL_{s}$ is $\chs_s$-stable. Thus, by Theorem \ref{MainThm}, $\cL_s$ is isomorphic to a special Lagrangian sphere $L_{\chX_s} \subset \chX$. According to \cite{PolishZaslow}, the mirror object to $\cE_C$ is always isomorphic to a special Lagrangian $L_{\chC} \subset \chC$. We conclude that $L_{\chX_s} \times L_{\chC} \subset \chZ_s$ is a special Lagrangian isomorphic to $\cL_{\chZ_s}$. This proves our claim.   
\end{proof}
\section{Compatibility with real Fourier-Mukai transform}\label{dHYMSec}

Fix $(X, s(\bB + \ii \omega), \Omega)$, with sublattice $H = \langle \sigma_0, E \rangle \subset \bL$, and its Hodge-theoretic mirror $(\chX_s, \chB_s + \ii \chom_s, \chOm_s)$. In particular $X$ and $\chX_s$ are both members of families approaching large complex structure limits. 

In the present Section we sketch a heuristic argument explaining how our construction should be compatible with an alternative, analytic approach to proving the implication \emph{stability $\Rightarrow$ sLag} precisely in the same context as for Theorem \ref{MainThm}. This is conjectural and we do not carry out the required analysis.

The proposed approach rests on glueing certain local models to SYZ special Lagrangian fibrations on $X, \chX$ constructed using the results of Gross-Wilson \cite{GrossWilson}. Applying the setup of  \cite{GrossWilson} requires some assumptions.
\begin{definition}\label{GrossWilsonAdmissDef} Following Definition \ref{GrossWilsonRotDef}, we denote by 
\begin{equation*}
X_K := (X_K, \omega_K, \Omega_K),\,\chX_K := (\chX_K, \chom_K, \chOm_K)
\end{equation*}
the corresponding hyperk\"ahler rotations with respect to the angle $\psi = 0$, so that we have
\begin{align*}
\Omega_{K} &= \Imm(\Omega)+ \ii s\omega,\, \omega_{K} = \Rea( \Omega),\,\chOm_{K} = \Imm(\chOm_s)+ \ii \chom_s,\, \chom_{K} = \Rea( \chOm_s). 
\end{align*}  
at the level of forms. We say that $X, \chX$ are \emph{admissible} if $X_K, \chX_K$ admit holomorphic elliptic fibrations 
$p_{X_K}\!: X_K \to \PP^1,\,p_{\chX_K}\!: \chX_K \to \PP^1$ with fibre class $E$ (under the given identifications $\bL \cong H_2(X_K, \Z)$, $\bL \cong H_2(\chX_K, \Z)$), which are generic (i.e. they have precisely 24 singular fibres of type $I_1$).
\end{definition}

Suppose that $X, \chX$ are admissible in the sense of Definition \ref{GrossWilsonAdmissDef}. By the results of Gross and Wilson \cite{GrossWilson} concerning the SYZ conjecture for K3 surfaces, for any open $U \subset S^2$ away from a $0$-dimensional analytic (discriminant) locus $\Delta \subset S^2$, there exist special Lagrangian torus fibrations
\begin{equation*}
f_U\!: X |_U \to U,\,\chf_U\!: \chX_U \to U  
\end{equation*}
(obtained by hyperk\"ahler rotation of $p_{X_K}$, $p_{\chX_K}$) which are arbitrarily close in $C^{\infty}$ to suitable semi-flat dual local models 
\begin{equation*}
f^{0}_U\!: (X |_U, \omega_U, \Omega_U) \to U,\,\chf^{0}_U\!: (\chX_U, \chom_U, \chOm_U) \to U  
\end{equation*}
as described by Leung-Yau-Zaslow in their classic work \cite{LeungYauZaslow}.

Following the notation of the previous sections, we consider the \emph{dHYM equation} on $X$,
\begin{equation}\label{dHYMequ}
\Imm e^{-\ii \hat{\theta}} (s\omega + \ii (\sigma_0 + s\hat{\bB}) + \ii \del\delbar \varphi)^2 = 0, 
\end{equation}
to be solved for $\varphi \in C^{\infty}(X, \R)$, given fixed smooth $(1,1)$-forms representing $\omega$, $\sigma_0$, and $\hat{\bB}$, where the phase $e^{-\ii \hat{\theta}}$ is determined cohomologically by integration.   

It was observed by Jacob and Yau \cite{JacobYau_special_Lag} that, by a simple computation, using the notation of our Lemma \ref{TwistedPosLem}, \eqref{dHYMequ} is equivalent to the complex Monge-Amp\`ere equation
\begin{equation*}
\hat{\eta}^2 = \left(\sin(\hat{\theta})\left(\sigma_0 + s\hat{\bB} + \ii\del\delbar\varphi\right) + \cos(\hat{\theta}) s\omega\right)^2 = \omega^2. 
\end{equation*}
It follows from the Calabi-Yau theorem that if twisted ampleness $\hat{\eta} > 0$ holds then there exists a unique smooth $\varphi$ solving \eqref{dHYMequ}, up to an additive constant. 

Suppose now that $U$ is obtained from $S^2$ by removing small discs around the points of the discriminant $\Delta$. We can perform a \emph{real Fourier-Mukai transform}, in the sense of Leung-Yau-Zaslow \cite{LeungYauZaslow}, of the fixed solution $\varphi$, thought of as (the logarithm of) a Hermitian metric on the fibres of the trivial bundle $\olo_X$, with respect to the dual semi-flat local models $f^{0}_U$, $\chf^{0}_U$, yielding a smooth section $\tilde{L}_U \subset \chX_U$ of $\chf_U$. By the properties of the real Fourier-Mukai transform, as the special Lagrangian fibrations $f_U$, $\chf_U$ approach $f^{0}_U$, $\chf^{0}_U$ smoothly, $\tilde{L}_U \subset \chX_U$ becomes arbitrarily close in $C^{\infty}$ to being calibrated by $\Rea(\Omega_{\chX})$ (restricted to $\chX_U$). Thus we expect that by a more careful but relatively standard analysis, using the quantitive implicit function theorem, $\tilde{L}_U$ can be perturbed to a genuine sLag section $L_U \subset \chX_U$ of $\chf_U$. This should be essentially a special case of the more sophisticated results of Datar, Jacob and Zhang \cite{DatarJacob_collapsedK3} (their analysis is performed for Yang-Mills connections of arbitrary rank rather than rank $1$ deformed Yang-Mills). It is natural to regard the special Lagrangian section $L_U \subset \chX_U$ as the mirror of the structure sheaf $\olo_U$.

The most difficult step thus is \emph{glueing local model sLags at each point of $\Delta$ to $L_U$ in order to obtain a special Lagrangian $L \subset \chX$ lying in the correct Fukaya class $\cL$}. We conclude with a discussion of the expected local models. 

For this purpose, we need to recall one of Gross-Wilson's central results in more detail.
\begin{thm}[\cite{GrossWilson}, Theorems 4.5 and 5.6] Let $j\!: J \to \PP^1$ be an elliptically fibred K3 surface with (holomorphic) section and 24 singular fibres over $\Delta$. Then there exist open sets $U^i_1 \subset U^i_2 \subset \PP^1$, $i = 1,\ldots, 24$, each diffeomorphic to a disc, with $U^i_j \cap \Delta = \{p_i\}$, a positive constant $\varepsilon_0$ such that, for all $0 < \eps < \eps_0$, for any elliptic K3 $p\!: (M, \omega_M, \Omega_M) \to \PP^1$ with Jacobian $j\!: J \to \PP^1$ with holomorphic $2$-form $\Omega$ satisfying $ (\Rea\omega_M)^2 = (\Rea \Omega_J)^2$ in cohomology, and for any K\"ahler class $[\omega_{M,\eps }]$ on $M$ with
\begin{equation*}
[\omega_{M,\epsilon}].M_b = \eps,\,[\omega_{M,\epsilon}]^2 = (\Rea\omega_M)^2 = (\Imm\omega_M)^2,
\end{equation*} 
there exists a K\"ahler metric $\omega_{M, \eps} \in [\omega_{M,\eps}]$ with the following properties:
\begin{enumerate}
\item[$(i)$] $\omega_{M, \eps}|_{p^{-1}\big(\PP^1\setminus \bigcup_i U^i_2 \big)} $ is a semi-flat metric (not necessarily standard).
\item[$(ii)$] $\omega_{M, \eps}|_{p^{-1}U^i_1} = T^*_{\sigma_i}\omega_{OV}$, where $\omega_{OV}$ is a metric of Ooguri-Vafa type and $T_{\sigma_i}$ denotes translation (along the fibres) by a (not necessarily holomorphic) section.
\item[$(iii)$] the metrics $\omega_{M, \eps}$ become exponentially close in $C^{k, \alpha}$ to the unique Ricci-flat K\"ahler metric representing $[\omega_{M, \eps}]$ as $\eps \to 0$. 
\end{enumerate}
\end{thm}

Semi-flat and Ooguri-Vafa type metrics are discussed in detail in \cite{GrossWilson}, Sections 2 and 3. The latter in particular is a family of non-complete hyperk\"ahler metrics defined on the total space of a holomorphic Lefschetz fibration $p_Y\!: Y \to U^i_1$ with a single $I_1$ fibre.  

Applying the above result to $M = X_{K}$ and $M = \chX_{K}$, we find that the expected local model that should be glued to $L_U$ at a singular point $p_i$ is the image, under global hyperk\"ahler rotation from $M$ to $X$ or $\chX$, of a suitable section $s_i$ of $p_Y\!: Y \to U^i_1$, satisfying
$\Omega_M|_{p^{-1}U^i_1} . \pd(T^*_{\sigma_i} s_i) = 0$.
In other words, the candidate local models are local holomorphic sections such that $T^*_{\sigma_i} s_i$ are local holomorphic sections of $p_Y\!: Y \to U^i_1$, endowed with the complex structure induced by $\Omega_M|_{p^{-1}U^i_1}$. Note that local holomorphic sections $u_i$ always exist, and then we can simply set $s_i = T^*_{-\sigma_i} u_i$. 

Setting $U := \PP^1\setminus \bigcup_i U^i_2$, the problem then becomes showing that suitable local holomorphic sections $u_i$ can be chosen, and that the glueing to $L_U$ can be performed, in such a way that the Fukaya equivalence class of the glued special Lagrangian is mirror to $\olo_X$.  

\bibliographystyle{abbrv}
 \bibliography{biblio_dHYM}
 
\noindent SISSA, via Bonomea 265, 34136 Trieste, Italy;\\
Institute for Geometry and Physics (IGAP), via Beirut 2, 34151 Trieste, Italy\\
jstoppa@sissa.it

\end{document}